\documentclass[12pt]{amsart}

\usepackage{amssymb,epsfig,mathrsfs}
\usepackage{fullpage}
\usepackage{graphicx}
\usepackage{mathtools}
\usepackage{xspace}  

\usepackage{tikz-cd}

\def\Z{{   \mathbb Z }}
\def\Q{{\mathbb Q}}

\def\F{{\mathbb F}}

\def\FF{{\mathcal{F}}}
\def\ch{{\mathrm{ch}}}
\def\Freq{{\mathrm{Freq}}}

\def\Meas{{\mathrm{Meas}}}

\def\Ssg{Schur $\sigma$-group\xspace}
\def\Ssgs{Schur $\sigma$-groups\xspace}
\def\Ssag{Schur $\sigma$-ancestor group\xspace}
\def\Ssags{Schur $\sigma$-ancestor groups\xspace}

\def\Spsg{Schur$+1$ $\sigma$-group\xspace}
\def\Spsag{Schur$+1$ $\sigma$-ancestor group\xspace}

\def\Spsgs{Schur$+1$ $\sigma$-groups\xspace}
\def\Spsags{Schur$+1$ $\sigma$-ancestor groups\xspace}

\def\Measc{{\mathrm{Meas}_c}}

\def\Rbar{\overline{R}}

\def\magma{{\sc Magma}}
\def\pari{{\sc PARI/GP}}

\def\G{{G}}

\def\ab{{\mathrm{ab}}}

\def\Gal{\mathrm{Gal}}
\def\Cl{\mathrm{Cl}}

\def\Epi{{ \mathrm{Epi}}}

\def\Aut{{\mathrm{Aut}}}

\newtheorem{theorem}{Theorem}[section]
\newtheorem{conjecture}[theorem]{Conjecture}
\newtheorem{thm-fr}{Th\'eor\`eme}[section]
\newtheorem{lem-fr}{Lemme}[section]
\newtheorem{cor-fr}{Corollaire}[section]
\newtheorem{lemma}[theorem]{Lemma}

\newtheorem{prop-fr}[theorem]{Proposition}

\theoremstyle{remark}

\newtheorem{remark}[theorem]{Remark}
\newtheorem{rem-fr}[theorem]{Remarque}
\newtheorem{example}[theorem]{Example}

\newtheorem{definition}[theorem]{Definition}
\newtheorem{def-fr}[theorem]{D\'efinition}

\begin{document}

\title{Heuristics for $p$-class Towers of Real Quadratic Fields}

\author{Nigel Boston}
\address{Department of Mathematics, University of Wisconsin - Madison,  
480 Lincoln Drive, Madison, WI 53706, USA}
\email{boston@math.wisc.edu}

\author{Michael R. Bush}
\address{Department of Mathematics,
Washington and Lee University,
204 W.\ Washington Street,
Lexington, VA 24450, USA}
\email{bushm@wlu.edu}

\author{Farshid Hajir}
\address{Department of Mathematics \& Statistics, University of Massachusetts - Amherst, 710 N. Pleasant Street, Amherst, MA 01003, USA}
\email{hajir@math.umass.edu}

\thanks{The research of NB is supported by Simons grant MSN179747. The work of the MRB was partially supported by summer Lenfest Grants from Washington and Lee University.}


\begin{abstract}
{Let $p$ be an odd prime. For a number field $K$, we let $K_\infty$ be the maximal unramified pro-$p$ extension of $K$; we call the group $\Gal(K_\infty/K)$ the $p$-class tower group of $K$. In a previous work, as a non-abelian generalization of the work of Cohen and Lenstra on ideal class groups, we studied how likely it is that a given finite $p$-group occurs as the $p$-class tower group of an imaginary quadratic field. Here we do the same for an arbitrary real quadratic field $K$ as base.  As before, the action of $\Gal(K/\Q)$  on the $p$-class tower group of $K$ plays a crucial role; however, the presence of units of infinite order in the ground field significantly complicates the possibilities for groups that can occur. 
We also sharpen our results in the imaginary quadratic field case by removing a certain hypothesis, using ideas of Boston and Wood. In an appendix, we show how the probabilities introduced for finite $p$-groups can be extended in a consistent way to the infinite pro-$p$ groups which can arise in both the real and imaginary quadratic settings.}
\end{abstract}

\maketitle

\section{Introduction}

In the 1980s, Cohen and Lenstra gave a theoretical framework for the variation of class groups of quadratic fields.  
The Cohen-Lenstra idea is twofold: the first part is to identify, in any relevant number-theoretical situation, the correct collection of groups which can arise as the groups of number-theoretical interest; the second part is to define a natural measure or probability distribution on this collection.
The heuristic then is that the probability attached to the group in the identified collection is the same as the frequency of occurrence as a group of number-theoretical interest.

In \cite{BBH}, we initiated the study of a natural non-abelian extension of Cohen and Lenstra's work.  Fix an odd prime $p$.
For a quadratic field $K$, we consider the Galois group $G_K$ of the maximal unramified $p$-extension of $K$. Note that the maximal abelian quotient $G_K^{\ab}$ is isomorphic to the $p$-class group of $K$ by class field theory.  We will call $G_K$ the $p$-class tower group of $K$, since that is how it first arose in the 1930s in the work of Artin, Hasse, Furtwangler and others.  In \cite{BBH}, we treated the case of imaginary quadratic fields.  The content of \cite{BBH} included a) an identification of the ``right'' collection of groups (\Ssgs), b) an investigation of an associated measure giving the frequency of groups within that collection, and c) a numerical study of $p$-class tower groups of imaginary quadratic fields to test the conjecture we developed using a) and b).  In this work, we treat real quadratic fields in the same manner.  As is to be expected, the presence of units of infinite order in the base field has a marked influence on the structure of $G_K$, and this makes some aspects of the current work slightly more complicated and more interesting than in \cite{BBH}.  

The organization of the paper is as follows. In Section~2, we define \Spsgs and introduce certain measures for these groups and also special finite quotients which we call \Spsags. We note that defining measures for the latter happens first and is then used in defining measures for the former.  In Section~3, we state the main heuristic, to the effect that a finite \Spsg occurs as the $p$-class tower group of a real quadratic field with frequency according to its measure.  In Section~4, we introduce IPADs and their associated measures. This provides a way for us to indirectly test our conjectures since computing the full Galois group $G_K$ is usually difficult unless the group is small. We then compare our theoretical predictions with numerical data in Section~5. We close with an Appendix in which we address assigning measures to infinite groups, an issue we skirted around in~\cite{BBH} and also in the main body of this paper. This is primarily of theoretical interest since any kind of direct test of our conjectures in the context of infinite groups seems out of reach currently.


\section{Schur$+1$ $\sigma$-groups}

\subsection{Preliminaries}
Fix an odd prime $p$ and a positive integer $g$. Let $F$ be the free pro-$p$ group on $g$ generators $x_1,...,x_g$. 
For a pro-$p$ group $G$, recall that $d(G)=\dim_{\F_p} H^1(G,\F_p)$ and $r(G)=\dim_{\F_p} H^2(G,\F_p)$ are its minimal number of (topological) generators and relations, respectively.  

In \cite{KV}, Koch and Venkov defined the notion of a  Schur $\sigma$-group of rank $g$. We recall its definition.  

\begin{definition}
A GI-automorphism (Generator-Inverting automorphism) of $G$ is an element $\sigma \in \Aut(G)$ of order $2$ such that $\sigma$ acts as inversion on $G^{\mathrm{ab}}$.  
\end{definition}

\begin{definition}
A finitely presented pro-$p$ group $G$ is called a {\em  Schur $\sigma$-group of rank $g$} if it satisfies:
1) $G^{\ab}$ is finite; 2) $(d(G),r(G))=(g,g)$; and 3) there exists a GI-automorphism
$\sigma$ of $G$. 
\end{definition}

Koch and Venkov were motivated to make the above definition through their study of the properties of the Galois group of the maximal unramified $p$-extension of an imaginary quadratic field.  Similar considerations for real quadratic fields lead us to the following definition.

\begin{definition}
A finitely presented pro-$p$ group $G$ is called a {\em  \Spsg of rank $g$} if it satisfies:
1) $G^{\ab}$ is finite; 2) $(d(G), r(G))=(g,g)$ or $(g,g+1)$; and 3) there exists a GI-automorphism
$\sigma$ of $G$ which acts as inversion on $H^2(G,\F_p)$.
\end{definition}

\begin{remark}
We note that an automorphism $\sigma \in \Aut(G)$ of order 2 is a GI-automorphism if and only if it acts by inversion on $H^1(G,\F_p)$ (see \cite{Boston-inventiones}).  Thus an alternative, perhaps more natural, formulation of 3) above is: $3'$) there exists $\sigma \in \Aut(G)$ of order 2 which acts as inversion on 
$H^i(G,\F_p)$ for $i=1,2$.
\end{remark}

\begin{lemma}
The Galois group $G_K=\Gal(L/K)$ of the maximal unramified $p$-extension $L$ of a real quadratic field $K$ is a  \Spsg of rank $g$ where $g$ is the $p$-rank of the class group of $K$.
\end{lemma}
\begin{proof}
To ease the notation slightly, let us put $G=G_K$.
When $G$ is a finite $p$-group, this result has been observed by Schoof \cite{Schoof} (see especially Lemma 4.1). By working with appropriate cohomology groups, as in the work of Kisilevsky and Labute~\cite{KL} for extensions of CM fields, we now show the statement still holds in the infinite case.  

For condition 1), by class field theory, $G^{\ab}$ is isomorphic to the $p$-class group of $K$, hence it is finite.  Furthermore, the generator rank of $G$ is equal to the generator rank of this abelian $p$-group which gives part of condition 2). 
The  relation rank part of condition 2) comes from the fundamental estimate of Shafarevich (see \cite{Koch}, \cite{NSW}, or \cite{Shafarevich}), for the partial Euler characteristic of $G$, namely
\[
0\leq r(G)-d(G) \leq d(U_K/U_K^p),
\]
where $U_K=\mathcal{O}_K^\times$ is the unit group of $K$. We note that $d(U_K/U_K^p)=1$ by Dirichlet's Unit Theorem.   That a lift $\sigma$
of the non-trivial element of $\Gal(K/\Q)$ to the $p$-Hilbert class field of $K$ acts on $G_K^{\ab}$ by inversion can be seen from the fact that, if it did not, then it would have to act on some quotient of $G^{\ab}$ trivially, leading to an unramified $p$-extension of $\Q$, which does not exist.  Alternatively, by Artin reciprocity, the action of this lift of $\sigma$ on $G_K^{\ab}$ can be read off from the action of $\sigma$ on the ideal class group of $K$, which is via inversion because for any integral ideal $\frak{a}$ of $K$,  the product $\frak{a} \frak{a}^\sigma$ is principal, being an ideal of $\Z$. 

Let $\Delta$ and $\Gamma$ be the Galois groups of $K/\Q$ and $L/\Q$ respectively. We then have $\Delta \cong \Gamma/G$. The action of $\Delta$ on $H^2(G,\F_p)$ can be understood using the existence of a $\Delta$-equivariant injection
\[ H^2(G,\F_p) \hookrightarrow \hom(V/(K^\times)^p, \mathbb{F}_p) \cong V/(K^\times)^p \]
where $V$ consists of the elements $a \in K^\times$ satisfying $\langle a \rangle = \mathfrak A^p$ for some fractional ideal $\mathfrak{A}$ in $\mathcal{O}_k$.
This follows from work of Shafarevich. See~\cite[Section X.7]{KL} for a more detailed description. 

We also have an exact sequence
\[  1 \rightarrow U_K/U_K^p \rightarrow V/(K^\times)^p \rightarrow  \Cl_K[p] \rightarrow 1 \]
where $\Cl_K[p]$ consists of the ideal classes in $\Cl_K$ of order $p$ (see~\cite[Section~11.2]{Koch}). It is straightforward to verify that all of the maps above are $\Delta$-equivariant. Since $p$ is odd, the sequence splits and we obtain an $\mathbb{F}_p[\Delta]$-module injection
\[  H^2(G,\F_p)  \hookrightarrow U_K/U_K^p \oplus \Cl_K[p]. \]
We have already noted that $\sigma$ acts as inversion on $\Cl_K$ and so as inversion on the subgroup $\Cl_K[p]$.  We also see that $\sigma$ acts as inversion on $U_K/U_K^p$ since for all $u\in U_K$, we have $u u^\sigma \in \{\pm 1\} \subseteq U_K^p$ thanks to $p$ being odd.  Hence, $\sigma$ acts as inversion on $H^2(G,\F_p)$. 
\end{proof}

Given a finitely presented pro-$p$ group $G$ with GI-automorphism $\sigma$, we define
\[ X(G, \sigma) = \{s \in G \mid \sigma(s)=s^{-1}\} \]
and
\[ Y(G, \sigma) = \{s \in G \mid \sigma(s)=s\}. \]
As noted in~\cite{BBH}, different choices of the GI-automorphism are always conjugate, hence the sizes of these sets depends only on $G$. We let $y(G) = |Y(G, \sigma)|$.
For any such $G$, one can always find a presentation in which the generators lie in $X(F,\sigma)$ and the relations lie in $X = X(\Phi(F), \sigma)$ where $\Phi(F)$ is the Frattini subgroup of $F$ and $\sigma$ is an automorphism on $F$ that inverts the generating set. i.e. $\sigma(x_i) = x_i^{-1}$ for all $i$.  In general, when we refer to the GI-automorphism of a free group $F$ we shall always mean this particular GI-automorphism.

As in~\cite{BBH}, we will be working with certain special quotients of   \Spsgs by terms in a central series whose definition we now recall. For a pro-$p$ group $G$,  the {\em lower $p$-central series of $G$} is defined by  $P_0(G) = G$ and $P_{n+1}(G) = [G,P_n(G)]P_n(G)^p$ for $n \geq 0$. If $P_n(G) = 1$ for $n = c$ but not $n < c$ then we say that $G$ has $p$-class $c$. We use the notation $G_c$ to denote the quotient $G/P_c(G)$ which we call the {\em maximal $p$-class $c$ quotient of $G$}. 
If $G$ is a finitely generated pro-$p$ group then $G_c$ is a finite $p$-group. The $p$-class of $G_c$ is at most $c$ but is not necessarily equal to $c$. Equality holds if and only if $G$ has $p$-class at least $c$. 

The subgroups in the lower $p$-central series are characteristic so any GI-automorphism $\sigma$ on $G$ induces a GI-automorphism on the quotient $G_c$ for all $c \geq 1$. In particular, the GI-automorphism $\sigma$ on $F$ induces a GI-automorphism on $F_c$ which we will also denote $\sigma$. We let
\[ X_c = X(\Phi(F_c),\sigma) = \{s \in \Phi(F_c)  \mid  \sigma(s) = s^{-1} \}. \]

If $G$ and $H$ are pro-p groups with $H$ of $p$-class $c$ and $G_c \cong H$ then we say that $G$ is a {\em descendant of $H$} (or that $H$ is an {\em ancestor of $G$}). If $G$ has $p$-class $c+1$ then we say that $G$ is an {\em immediate descendant}. If $G$ is a  \Spsg then we refer to every finite quotient $G_c$ as a {\em  \Spsag}. Note that a finite \Spsg $G$ will itself be referred to as a \Spsag since $P_c(G) = 1$ and $G_c = G$ once $c$ is sufficiently large.

In~\cite{O}, an algorithm is described for enumerating all immediate descendants of a given finite $p$-group and we make use of this in Section~\ref{section-ipad-predictions}. Various quantities related to the algorithm are defined in terms  of abstract presentations in~\cite{O}. There are no problems however if one chooses to work with pro-$p$ presentations as we do. For further discussion of this point see~\cite[Remark~2.4]{BBH}.


\subsection{Measures on $p$-groups}

Let $G$ be a finite $p$-group of $p$-class $c$ with $d(G) = g$ and $r(G) = g$ or $g+1$. One can see that $G$ is a quotient of $F_{c'}$ for all $c' \geq c$. We will say that the tuple of elements $v = (t_1,\ldots,t_{g+1}) \in\Phi(F_{c'})^{g+1}$ {\em presents $G$} if $F_{c'} / \langle v \rangle \cong G$ where $ \langle v \rangle$ denotes the closed normal subgroup of $F_{c'}$ generated by $t_1,\ldots,t_{g+1}$.  We let $S_{c'} = S_{c'}(G)$ denote the set of all such tuples in $\Phi(F_{c'})^{g+1}$. 

When $G$ is a \Spsag we wish to consider tuples of relations satisfying an additional restriction. To explain this further, we need two lemmas. For the proofs of these lemmas and several other results in this section we will simply refer to~\cite{BBH} since the presence of an extra relation in the tuple has no effect on the arguments.

\begin{lemma}\label{phi-map1}
For all  $d \geq 1$, we have $X_d = X_d'$ where 
\[ X_d' =  \{t^{-1}\sigma(t) \mid  t \in \Phi(F_d)\}. \]
Hence, for all $g \geq 1$, the map $\phi_{d}:
\Phi(F_{d})^{g+1} \rightarrow X_{d}^{g+1}$ defined by 
\[(t_1,\ldots,t_{g+1}) \mapsto (t_1^{-1} \sigma(t_1),\ldots, t_{g+1}^{-1}\sigma(t_{g+1}))\] 
is surjective.  Indeed, for each  $w
  \in X_{d}^{g+1}$, the fiber $\phi_{d}^{-1}(w)$ is a coset of $Y_{d}^{g+1}$ in $\Phi(F_{d})^{g+1}$ where $Y_{d} = Y(F_{d},\sigma)$.
\end{lemma}
\begin{proof}
See~\cite[Lemma 2.5]{BBH}. 
\end{proof}

\begin{remark}\label{remark_phi_map}
As explained in~\cite[Remark 2.6]{BBH}, one consequence of this lemma is that $X = X'$ where
\[ X' = \{ t^{-1} \sigma(t) \mid t \in \Phi(F) \}. \]
It follows that the map $\phi: \Phi(F)^{g+1} \rightarrow X^{g+1}$ defined by $t \mapsto t^{-1} \sigma(t)$ in each component is surjective.
\end{remark}

\begin{lemma}\label{lemma_specialpres}
If $H$ is a  \Spsag of $p$-class $c$ then it can be presented with a tuple
of relations in $X_{c}^{g+1} \subseteq \Phi(F_{c})^{g+1}$. Conversely, the group $H$ presented by any tuple of relations in $X_{c}^{g+1}$  is a
\Spsag of $p$-class at most $c$.
\end{lemma}

\begin{proof}
See~\cite[Lemma 2.7]{BBH}. 
\end{proof}

\begin{definition}
Let $G$ be a \Spsag of $p$-class $c$ and generator rank $g$. For  $c' \geq c$, let $T_{c'} = T_{c'}(G)$ denote the set of all tuples in $X_{c'}^{g+1}$ which present $G$.  We then define the {\em $p$-class $c'$-measure of $G$} by
\[ \Meas_{c'}(G) = \frac{|T_{c'}|}{|X_{c'}|^{g+1}}.  \]
\end{definition}

\begin{remark}\label{remark-notation}
We have chosen to adopt the same notation as in~\cite{BBH}. This will not cause confusion unless one wants to discuss the values of both types of measure side by side. In this situation, one might use the notation $\Meas_{c'}^{+1}(G)$  to indicate that one is working with $(g+1)$-tuples of relations rather than $g$-tuples.
\end{remark}

 \begin{example}\label{ex-p3g2c2}
Let $p = 3$ and $g = c = 2$. In this case, $F_2 = F/P_2(F)$ has order $3^5$ and  the set $X_2$ happens to be a central elementary abelian subgroup of order $9$. In~\cite[Example~2.9]{BBH} we saw that there were exactly three Schur $\sigma$-ancestor groups of $3$-class $2$ and we computed their measures, as defined there, by explicitly counting tuples in $X_2$. These  groups are also \Spsags and it is not hard to see that no others exist of $3$-class $2$. Indeed, by enumerating $3$-tuples we can see that the list is complete and also compute their measures as \Spsags. 

One observes that $624$ of the $9^3 = 729$ $3$-tuples generate $X_2$ and give rise (after taking the quotient) to a group of order $27$ which we labeled $G_1$ in~\cite{BBH}.
Of the tuples that remain, $104$ generate one of the four subgroups of order $3$ inside $X_2$ and give rise to a group of order $81$ that we labeled $G_2$. 
That leaves $1$ tuple with all components trivial that  gives rise to the group $G_3 \cong F_2$.
It follows that the $p$-class $2$-measures of these $3$-groups as \Spsags are $\Meas_2(G_1) = 624/729$, $\Meas_2(G_2) = 104/729$ and
$\Meas_2(G_3) = 1/729$. 
\end{example}

As in ~\cite{BBH}, Lemma~\ref{lemma_specialpres} implies that
$\Meas_c(G)$ defines a discrete probability measure on the set of isomorphism classes of
maximal $p$-class $c$ quotients of all \Spsgs of generator rank $g$. This set is finite and consists of the \Spsags of $p$-class exactly $c$, together with all
\Spsgs of $p$-class less than $c$. The next theorem shows how these different probability measures are related.

\begin{theorem}\label{thm-meas-relations}
Let $G$ be a \Spsag of $p$-class $c$.
\begin{itemize}
\item[(i)]  We have
\[  \Meas_c(G) = \Meas_{c+1}(G) + \sum_Q \Meas_{c+1}(Q) \]
where the summation is over all immediate descendants $Q$ of $G$ which are \Spsags.
\item[(ii)] $\Meas_{c'}(G) = \Meas_{c+1}(G)$ for all $c' \geq c + 1$.
\end{itemize}
\end{theorem}
\begin{proof}
See~\cite[Theorem~2.11]{BBH}. 
\end{proof}

\begin{definition}\label{def-meas}
Let $G$ be a \Spsag  of $p$-class $c$. We define the {\em measure of $G$} (denoted $\Meas(G)$) to be the constant value of 
$\Meas_{c'}(G)$ for $c' \geq c+1$.
\end{definition}

\begin{remark}\label{remark-meas-zero}
As for Schur $\sigma$-groups, a finite $p$-group $G$ is a \Spsg if and only if $\Meas(G) > 0$. 
As opposed to the situation for Schur $\sigma$-groups, it is possible for a group $G$ to be both a \Spsg as well as a proper quotient of a larger \Spsg. This means that if $G$ has $p$-class $c$ then it is possible for both $\Meas(G) = \Meas_{c+1}(G)$ and the summation appearing on the right in Theorem~\ref{thm-meas-relations}(i) to be nonzero. 
\end{remark}


\subsection{Measures of abelian $p$-groups} 
\label{subsection-ab-measure}

We now define measures on certain collections of finite abelian $p$-groups. This will allow us to demonstrate that the heuristics introduced in Section~\ref{section-conjectures} are consistent with the Cohen-Lenstra heuristics for $p$-class groups of real quadratic fields.

Every abelian pro-$p$ group $G$ comes equipped with a unique GI-automorphism, namely the inversion mapping $x \mapsto x^{-1}$.  Consider the abelianizations $F^{\ab}$ and $F_c^{\ab}$. We define sets $X^{\ab}$ and $X_c^{\ab}$ in an analogous way to $X$ and $X_c$  but things are now simpler and it is easy to verify that
$X^{\ab} = \Phi(F^{\ab})$ and $X_c^{\ab} = \Phi(F_c^{\ab})$. 

Let $G$ be a finite abelian $p$-group of $p$-class $c$ with generator rank $g$ and let $c' \geq c$.  We will say that the tuple of elements 
$v = (t_1,\ldots,t_{g+1}) \in\Phi(F_{c'}^{\ab})^{g+1}$ {\em presents $G$} if $F_{c'}^{\ab} / \langle v \rangle \cong G$ where $ \langle v \rangle$ 
denotes the (normal) subgroup of $F_{c'}^{\ab}$ generated by $t_1,\ldots,t_{g+1}$. Such tuples must exist since $G$ is finite. We let 
$S_{c'}^{\ab} = S_{c'}^{\ab}(G)$ denote the set of all such tuples in $\Phi(F_{c'}^{\ab})^{g+1}$.   In the non-abelian setting, we introduced a second set of tuples $T_{c'} \subseteq S_{c'}$. We can do the same in the abelian setting, but the situation now is simpler and we have $T_{c'}^{\ab}$ = $S_{c'}^{\ab}$ since $X_{c'}^{\ab} = \Phi(F_{c'}^{\ab})$.

\begin{definition}
Let $G$ be an abelian $p$-group of $p$-class $c$ and generator rank $g$. For $c' \geq c$, we define the {\em abelian$+1$ $c'$-measure of $G$} by
\[ \Meas_{c'}^{\ab}(G) = \frac{|T_{c'}^{\ab}|}{|X_{c'}^{\ab}|^{g+1}}  \left(=  \frac{|S_{c'}^{\ab}|}{|\Phi(F_{c'}^{\ab})|^{g+1}}\right).  \]
Remark~\ref{remark-notation} also applies to our choice of notation in the abelian setting.
\end{definition}

The remaining results in this section allow us to define the quantity $\Meas^{\ab}(G)$ for a finite abelian $p$-group $G$,  to relate the abelian measures to the non-abelian measures introduced in the previous section and to give explicit formulas for these measures. 

\begin{theorem}\label{thm-meas-relations-ab}
Let $G$ be an abelian $p$-group of $p$-class $c$.
\begin{itemize}
\item[(i)]  We have
\[  \Meas_c^{\ab}(G) = \Meas_{c+1}^{\ab}(G) + \sum_Q \Meas_{c+1}^{\ab}(Q) \]
where the summation is over all immediate abelian descendants $Q$ of $G$.
\item[(ii)] $\Meas_{c'}^{\ab}(G) = \Meas_{c+1}^{\ab}(G)$ for all $c' \geq c + 1$.
\end{itemize}
\end{theorem}
\begin{proof}
Proved in a similar fashion to Theorem~\ref{thm-meas-relations}.
\end{proof}

\begin{definition} \label{def-meas-ab}
Let $G$ be an abelian $p$-group of $p$-class $c$. We define the {\em abelian$+1$ measure of $G$} (denoted $\Meas^{\ab}(G)$) to be the constant value of $\Meas_{c'}^{\ab}(G)$ for $c' \geq c+1$.
\end{definition}
\begin{remark}
It follows from part (i) of Theorem~\ref{thm-meas-relations-ab} that if $G$ is an abelian $p$-group of $p$-class $c$ then
\[ \Meas^{\ab}(G) = \Meas_c^{\ab}(G) - \sum_Q \Meas_{c+1}^{\ab}(Q) \]
where the summation is over all abelian groups $Q$ of $p$-class $c+1$ with $Q/Q^{p^c} \cong G$;  here $Q^{p^c}$ is the subgroup of $Q$ generated by all $p^c$-th powers.
\end{remark}

\begin{theorem}\label{thm-meas-vs-meas-ab}
Let $G$ be an abelian $p$-group of $p$-class $c$. For all $c' \geq c$ we have
\[ \Meas_{c'}^{\ab}(G) = \sum_Q \Meas_{c'}(Q) \]
where the summation is over all \Spsags $Q$ with $p$-class at most $c'$ and $Q^{\ab} \cong G$.
\end{theorem}
\begin{proof}
See~\cite[Theorem~2.18 ]{BBH}.
\end{proof}

The next result is very similar to Theorem~2.20 in \cite{BBH}. However, observe that an extra factor $|G|$ now appears in the denominators and there has been a slight change to the indexing on one of the products. 
\begin{theorem}\label{main-group-theory-ab}
Let $G$ be an abelian $p$-group of $p$-class $c$ and generator rank $g$. We have
\begin{equation*}
\Meas_c^{\ab}(G)  
= \frac{1}{|\Aut(G)||G|} \,
p^{g(g+1)}
\prod_{k=1}^g (1 - p^{-k}) \prod_{k={g + 2 - u}}^{g+1} (1 - p^{-k})
\end{equation*}
where $u$ counts the number of cyclic groups of order strictly less than $p^c$ in the direct product decomposition of $G$. \\
For $c' > c$, we have
\begin{eqnarray*}
\Meas_{c'}^{\ab}(G) = \Meas^{\ab}(G) = 
\frac{1}{|\Aut(G)||G|}
\,
p^{g(g+1)}
\prod_{k=1}^g (1 - p^{-k}) \prod_{k=2}^{g+1} (1 - p^{-k}).
\end{eqnarray*}
\end{theorem}
\begin{proof}
As in the proof of~\cite[Theorem 2.20]{BBH}, the number of normal subgroups $\Rbar$ such that  $F_c^{\ab}/\Rbar \cong G$ 
is
\[ \frac{|\mathrm{Epi}(F,G)|}{|\Aut(G)|} =\frac{ |\Phi(G)|^g}{|\Aut(G)|} \prod_{k=1}^g (p^g - p^{g-k}). \]
The change occurs in the next step. The number of $(g+1)$-tuples that generate each subgroup $\Rbar$ is
\[   |\Phi(\Rbar)|^{g+1} \prod_{k=1}^u (p^{g+1} - p^{u-k}). \]
Combining the statements above, we have
\begin{eqnarray*}
\mathrm{Meas}_c^{\ab}(G) = \frac{|S_c^{\ab}(G)|}{|\Phi(F_c^{\ab})|^{g+1}} 
&=& \frac{1}{|\Phi(F_c^{\ab})|^{g+1}}  \frac{ |\Phi(G)|^g}{|\Aut(G)|} \prod_{k=1}^g (p^g - p^{g-k}) |\Phi(\Rbar)|^{g+1} \prod_{k=1}^u (p^{g+1} - p^{u-k}) \\
&=& \frac{1}{|\Phi(F_c^{\ab})|^{g+1}}  \frac{ (|\Phi(F_c^{\ab})|/|\Rbar|)^g}{|\Aut(G)|} \prod_{k=1}^g (p^g - p^{g-k}) \frac{|\Rbar|^{g+1}}{p^{(g+1)u}} \prod_{k=1}^u (p^{g+1} - p^{u-k}) \\
&=& \frac{|\Rbar|}{|\Phi(F_c^{\ab})|}  \frac{ 1}{|\Aut(G)|} \prod_{k=1}^g (p^g - p^{g-k}) \frac{1}{p^{(g+1)u}} \prod_{k=1}^u (p^{g+1} - p^{u-k}) \\
&=& \frac{p^g}{|G|}  \frac{ 1}{|\Aut(G)|} \prod_{k=1}^g (p^g - p^{g-k}) \frac{1}{p^{(g+1)u}} \prod_{k=1}^u (p^{g+1} - p^{u-k}) \\
&=& \frac{1}{|\Aut(G)||G|} \,
p^{g(g+1)}
\prod_{k=1}^g (1 - p^{-k}) \prod_{k={g + 2 - u}}^{g+1} (1 - p^{-k}).
\end{eqnarray*}
The derivation of the formula for $\Meas_{c'}^{\ab}(G)$ for $c' > c$ involves replacing $u$ with $g$ as discussed at the end of the proof of~\cite[Theorem~2.20]{BBH}.
\end{proof}

\begin{remark}\label{remark-eta}
If we define $\eta_j(p) = \prod_{k=1}^j (1 - p^{-k})$ as in~\cite{CL2}, then the formulas in Theorem~\ref{main-group-theory-ab} can be written
\begin{eqnarray*}
\mathrm{Meas}_c^{\ab}(G) &=& 
\frac{1}{|\Aut(G)| |G|} \,
p^{g(g+1)} 
\left(\frac{ \eta_g(p) \eta_{g+1}(p)}{\eta_{g + 1 - u}(p) }\right) \\
\mathrm{Meas}^{\ab}(G) &=& \frac{1}{|\Aut(G)||G|}
\,
p^{g(g+1)} 
\left(\frac{ \eta_g(p) \eta_{g+1}(p)}{ 1 - p^{-1}}\right).
\end{eqnarray*}
\end{remark}


\subsection{Formula for $\Meas_c(G)$}
\label{subsection-formulas}

In this section we give explicit formulas for $\Measc(G)$ and $\Meas(G)$ where $G$ is a \Spsag. In the context of \Ssgs, we obtained such formulas \cite[Theorem~2.25]{BBH} under an additional technical hypothesis which we called KIP. Subsequently, it was shown by Boston and Wood in \cite{BW} how to obtain such formulas without this assumption. We begin by explaining how this is carried out.

In \cite[Theorem~2.25]{BBH}, the group $G$ is assumed to be a \Ssag of $p$-class~$c$ and rank~$g$ satisfying KIP.  To derive a formula for $\Meas_c(G)$, we began by enumerating the normal subgroups $\Rbar$ in $F_c$ with $F_c/\Rbar \cong G$. There are 
\[ \frac{|\Epi(F,G)|}{|\Aut(G)|} =\frac{ |\Phi(G)|^g}{|\Aut(G)|} \prod_{k=1}^g (p^g - p^{g-k}) = \frac{ |G|^g}{|\Aut(G)|} \prod_{k=1}^g (1 - p^{-k})\]
such subgroups where $\Epi(F,G)$ denotes the set of surjective homomorphisms from $F$ to $G$. We then showed that each $\Rbar$ is generated as a normal subgroup of $F_c$ by the {\em same} number of tuples in $\Phi(F_c)^g$. This allowed us to compute $|S_c(G)|$ by summing this number over all such subgroups $\Rbar$ and so lead to a formula for the ratio $|S_c(G)|/ |\Phi(F_c)|^g$. KIP then entered the picture through an application of \cite[Lemma~2.23]{BBH} which allowed us to convert this into a formula for the desired ratio
\[  \Measc(G) = \frac{|T_c(G)|}{|X_c|^g}. \]
The resulting formula for $\Measc(G)$ involved the quantity $|\Aut(G)|$. In  \cite[Corollary~2.30]{BBH}, we gave a formula in terms of $|\Aut_\sigma(G)|$. The second conversion was carried out using \cite[Theorem~2.29]{BBH} whose proof involved another application of the KIP assumption. The formula for $\Meas(G)$ was obtained in a similar fashion.

The key observation in  \cite[Section~4]{BW} is that one can directly compute the quantity $|T_c(G)|$ by enumerating over normal subgroups $\Rbar$ in $F_c$ which are $\sigma$-invariant and then counting the number of tuples in $X_c^g$ that generate each such subgroup. One then immediately obtains a formula for $\Measc(G)$ in terms of $|\Aut_\sigma(G)|$. This is the approach taken in the proof below noting that, in the context of \Spsags, our tuples of relations now have $g+1$ components.

To avoid any confusion over notation\footnote{The objects that we denote $X(G)$ and $Y(G)$ in this paper and also~\cite{BBH} are denoted $Y(G)$ and $Z(G)$ respectively in \cite{BW}}, 
we briefly recall some useful results from \cite[Section~4]{BW} involving the sizes of the sets $X(G) = X(G,\sigma)$ and $Y(G) = Y(G, \sigma)$ that will be used several times in what follows. In these results, the groups involved are understood to be finite $p$-groups, each with a specified GI-automorphism~$\sigma$. Given such a group $G$, \cite[Lemma~4.2]{BW} states that $|G| = |X(G)||Y(G)|$. The authors observe that this follows directly from \cite[Theorem~3.5]{Gor}. Given a short exact sequence $1 \rightarrow K \rightarrow G \rightarrow H \rightarrow 1$ of such groups in which the maps are $\sigma$-equivariant, \cite[Lemma~4.3 and 4.4]{BW} show that the induced map $Y(G) \rightarrow Y(H)$ is surjective from which it follows that $|Y(G)| = |Y(K)| |Y(H)|$. Since $|G| = |K| |H|$, one can then see that we also have $|X(G)| = |X(K)| |X(H)|$. The surjectivity of the map is established using the Schur-Zassenhaus Theorem. It also follows from a more elementary argument similar to that used in the proof of part~(iv) of~\cite[Lemma~2.23]{BBH}.

Finally, recall that given a presentation $F/R$ for $G$, we can form the $\mathbb{F}_p$-vector space $R/R^\ast$ where $R^\ast$ is the closure of $R^p[F,R]$ in $F$. This is the {\em $p$-multiplicator of $G$} and its dimension is equal to the relation rank $r(G)$. If $G$ has $p$-class $c$, then the subspace $R^\ast P_c(F)/R^\ast$ is called the {\em nucleus of $G$}. We define $h(G)$ to be the difference between the dimensions of the $p$-multiplicator and the nucleus. Equivalently, it is the dimension of the $\mathbb{F}_p$-vector space $R/R^\ast P_c(F)$. For more discussion of these quantities, see~\cite[Section~2.1]{BBH} and the remarks immediately following \cite[Definition~2.24]{BBH}.

\begin{theorem}\label{thm-main-formula}
Let $G$ be a \Spsag of $p$-class $c$ and rank
$g$ and let $h=h(G)$. Then
we have

\[ 
\Meas_c(G) = 
\frac{y(G)}{|\Aut_\sigma(G)| |G|} \,
p^{g(g+1)}
\left(\frac{ \eta_g(p) \eta_{g+1}(p)}{\eta_{g + 1 - h}(p) }\right).
\]
If $G$ is also a \Spsg with $r = r(G)$ then
\[
\Meas(G) = \frac{y(G)}{|\Aut_\sigma(G)||G|}
\,
p^{g(g+1)} 
\left( \frac{ \eta_g(p) \eta_{g+1}(p)}{\eta_{g + 1 - r}(p) } \right),
\]
otherwise $\Meas(G) = 0$.
\end{theorem}
 
\begin{proof}
Each tuple of elements in $T_c(G) \subseteq X_c^{g+1}$ generates a normal subgroup $\Rbar$ of $F_c$ which is $\sigma$-invariant since $\sigma(X_c) = X_c$. It is straightforward to show that the number of such $\sigma$-invariant normal subgroups is $|\Epi_\sigma(F,G)| / |\Aut_\sigma(G)|$ where $\Epi_\sigma(F,G)$ denotes the set of surjective $\sigma$-equivariant homomorphisms from $F$ to $G$. By \cite[Lemma~4.6]{BW}, we have
\[  \frac{|\Epi_\sigma(F,G)|}{|\Aut_\sigma(G)|} = \frac{|X(G)|^g}{|\Aut_\sigma(G)|} \prod_{k=1}^g (1 - p^{-k}). \]

We now show that each $\Rbar$ is generated as a normal subgroup of $F_c$ by the {\em same} number of tuples in $X_c^{g+1}$.
Observe that a $(g+1)$-tuple of elements generates $\Rbar$ as a normal subgroup of $F_c$ if and only if its image generates the $h$-dimensional $\F_p$-vector space  $\Rbar/\Rbar^\ast$ where $\Rbar^\ast = \Rbar^p [F_c,\Rbar] = P_c(F)R^\ast/P_c(F)$ with $R$ the preimage of $\Rbar$ in $F$.
We also note that since $G = F_c/\Rbar$ is a \Spsag, the induced action of $\sigma$ on  $\Rbar/\Rbar^\ast$ is by inversion.
Applying the same argument as in the proof of~\cite[Lemma~4.5]{BW}, we see that the intersection of $Y(\Rbar) = Y_c \cap \Rbar$ with each fiber of the reduction map $\Rbar \rightarrow \Rbar/\Rbar^\ast$ has constant size $|\Rbar^\ast|/|Y(\Rbar)|$. Since $\Rbar/\Rbar^\ast$ is $h$-dimensional, it follows that the number of tuples in $X_c^{g+1}$ which generate $\Rbar$ is
\[   \left(\frac{|\Rbar^\ast|}{|Y(\Rbar)|}\right)^{g+1} \prod_{k=1}^h (p^{g+1}-p^{h-k})  =  \frac{|\Rbar^\ast|^{g+1}}{|Y(\Rbar)|^{g+1}} p^{h(g+1)} \prod_{k=1}^h (1 -p^{h - g - 1 - k}). \]

Combining the above, we see that 
\begin{eqnarray*}
\Measc(G) = \frac{|T_c(G)|}{|X_c|^{g+1}} 
&=& \frac{1}{|X_c|^{g+1}} \frac{|X(G)|^g}{|\Aut_\sigma(G)|} \prod_{k=1}^g (1 - p^{-k}) \frac{|\Rbar^\ast|^{g+1}}{|Y(\Rbar)|^{g+1}} p^{h(g+1)} \prod_{k=1}^h (1 -p^{h - g - 1 - k}) \\
&=& \frac{1}{|\Aut_\sigma(G)|} \frac{|X(G)|^g |X(\Rbar)|^{g+1}}{|X_c|^{g+1}} \prod_{k=1}^g (1 - p^{-k}) \prod_{k=1}^h (1 -p^{h - g - 1 - k})
\end{eqnarray*}
where the simplification in the second line follows from the fact that $|\Rbar^\ast| = |\Rbar|/p^h$ and $|\Rbar| = |X(\Rbar)| |Y(\Rbar)|$. 

Now consider the sequences 
\[ 1 \rightarrow \Phi(F_c) \rightarrow F_c \rightarrow F_c/\Phi(F_c) \rightarrow 1 \quad \text{and} \quad
1 \rightarrow \Rbar \rightarrow F_c \rightarrow G \rightarrow 1. \] 
From the first, we deduce that 
\[ |X_c| = |X(\Phi(F_c))| = |X(F_c)|/|X(F_c/\Phi(F_c))| = |X(F_c)|/p^g \] 
since $\sigma$ acts by inversion on all of $F_c/\Phi(F_c)$. From the second, we have $|X(F_c)| = |X(G)| |X(\Rbar)|$ from which it follows that $|X_c| = |X(G)| |X(\Rbar)|/p^g$. Substitution then yields
\begin{eqnarray*}
\Measc(G) 
&=& \frac{1}{|\Aut_\sigma(G)| |X(G)|} \, p^{g(g+1)} \prod_{k=1}^g (1 - p^{-k}) \prod_{k=1}^h (1 -p^{h - g - 1 - k}) \\
&=& \frac{y(G)}{|\Aut_\sigma(G)| |G|} \, p^{g(g+1)} \prod_{k=1}^g (1 - p^{-k}) \prod_{k'=g + 2 - h}^{g+1} (1 -p^{-k'}) \\
&=& \frac{y(G)}{|\Aut_\sigma(G)| |G|} \, p^{g(g+1)} \left(\frac{ \eta_g(p) \eta_{g+1}(p)}{\eta_{g + 1 - h}(p) }\right) 
\end{eqnarray*}
where $y(G) = |Y(G)|$ and $\eta_j(p) = \prod_{k=1}^j (1 - p^{-k})$.

Now suppose $G$ is also a \Spsg. By definition, $\Meas(G) = \Meas_{c+1}(G)$. The latter can be evaluated by following the same steps as above, but noting that if $G = F/R$ then $P_c(F) \subseteq R$ since $G$ has $p$-class $c$ and this implies $P_{c+1}(F) \subseteq R^\ast$. It follows that $\Rbar/\Rbar^\ast \cong R/R^\ast$, where $\Rbar = R/P_{c+1}(F) \subseteq F_{c+1}$, and so has dimension $r = r(G)$. Thus one simply replaces $h$ with $r$ in the final formula above.

If $G$ is a \Spsag but not a \Spsg, then either $\sigma$ does not act by inversion on $R/R^\ast \cong \Rbar/\Rbar^\ast$ or the dimension $r$ of this space is larger than $g+1$. In either case, none of the tuples in $X_{c+1}^{g+1}$ can generate any of the invariant normal subgroups $\Rbar$ since this would lead to a contradiction. Thus $\Meas(G) = \Meas_{c+1}(G) = 0$.
\end{proof}

\subsection{Some numerical examples.} 
In this section, we illustrate the theory with some numerical examples in the simplest case where $p=3$ and $g=2$.  Let $G$ be a \Spsag of $3$-class $c$ and rank $2$. By Theorem~\ref{thm-main-formula},
\[ 
\Meas_c(G) = 
\frac{432 y(G)}{|\Aut_\sigma(G)| |G|} \,  \prod_{k=4 - h}^3 (1 -3^{-k}) 
\]
where $h = h(G)$. If $G$ is also a \Spsg, then we have
\[ 
\Meas(G) = 
\frac{432 y(G)}{|\Aut_\sigma(G)| |G|} \,  \prod_{k=4 - r}^3 (1 -3^{-k}) 
\]
where $r = r(G)$ is either $2$ or $3$.
\begin{itemize}
\item For $r=3$, the smallest examples of \Spsgs are, in the notation of \magma, SmallGroup$(81,i)$ for $i= 7,8,10$ with measures $1664/6561$, $1664/6561$, $3328/19683$  respectively, and SmallGroup$(243,i)$ for $i=16,18,19,20$ with measures $3328/177147$, $3328/177147,$ $ 1664/59049$, $1664/59049$ respectively. 
\item For $r=2$, the smallest examples of \Spsgs are the groups SmallGroup$(243,i)$ for $i=5,7$, with respective measures $1664/59049$ and $832/59049$.  As a point of comparison, the latter groups are also \Ssgs and would be assigned measures $128/729$ and $64/729$ respectively, in the context of~\cite{BBH}.  Thus our heuristics in the next section will predict different frequencies of occurrence for these groups as $G_K$ when $K$ is a real quadratic field. However, the $2:1$ ratio between the probabilities is preserved and so we expect SmallGroup$(243,5)$ to occur twice as frequently as SmallGroup(243,7) in both the real and imaginary quadratic settings. This is reflected in the available numerical data discussed in Section~\ref{section-data} below and also in \cite[Section 5]{BBH}. 
\end{itemize}

In the examples above, all of the groups $G$ involved have $3$-class $3$ and nuclear rank $0$. This implies $h(G) = r(G)$ and so $\Meas_3(G) = \Meas(G)$ in each case. In general, as noted in Remark~\ref{remark-meas-zero}, it is possible for $\Meas_c(G)$ and $\Meas(G)$ to both be nonzero and not equal. This occurs when $G$ is both a \Spsg and the ancestor of a strictly larger \Spsg. 

As an example of this new phenomenon, consider the group $G =$ SmallGroup$(729,8)$. It is a \Spsag with $3$-class 3, $r(G) = 3$ and $h(G) = 2$. By Theorem~\ref{thm-main-formula},  $\Meas_3(G) = 1664/531441$ and $\Meas(G) = \Meas_4(G) = 3328/1594323$.  By Theorem~\ref{thm-meas-relations}(i), the difference must equal the sum of $\Meas_4(Q)$ over all the \Ssags which are children of $G$. A computation shows that there are $3$ such children $Q$, all of which are \Spsags and each with $\Meas_4(Q) = 1664/4782969$. We then have
\[ 3 \cdot (1664/4782969) = 1664/1594323 = \Meas_3(G) - \Meas(G) \]
as expected. Interestingly, since none of the measures of the children could be omitted from the sum without contradicting the required equality in Theorem~\ref{thm-meas-relations}(i), we see that  we can actually use the theorem to deduce that all $3$ children must be \Spsags without an explicit check. Of course, one must still find a GI-automorphism for each group as a first step in order to be able to evaluate the formula and this can be costly in itself as the groups get larger.

When testing whether a group $G$ is a \Spsag, it is helpful to observe that  necessary conditions include possessing a GI-automorphism and having $h(G) \leq g + 1$ (for the latter condition see \cite{BLG} and \cite{BN}). These conditions do not suffice though and there exist groups which satisfy both, but are not \Spsags analogous to the pseudo-Schur groups of \cite{BBH}. We call such groups {\em pseudo-Schur+1 groups}. 

Returning to Example~\ref{ex-p3g2c2}, we see that the \Spsags $G_1$, $G_2$ and $G_3$ correspond to SmallGroup$(27,3)$, SmallGroup$(81,3)$ and SmallGroup$(243,2)$ respectively in \magma's database. All three have $3$-class $2$ and one computes that $r(G_1) = r(G_2) = 4 > 3$ and $r(G_3) = 5 > 3$ which means  these groups are not \Spsgs and thus we have $\Meas(G_i) = 0$ for $i = 1,2,3$. After computing the nuclear ranks and subtracting, we obtain $h(G_1) = 2$, $h(G_2) = 1$ and $h(G_3) = 0$. Using Theorem~\ref{thm-main-formula}, we see that $\Meas_2(G_1) = 208/243$, $\Meas_2(G_2) = 104/729$ and $\Meas_2(G_3) = 1/729$ and this agrees with the earlier values obtained by explicitly counting tuples of relations.

There are four other groups of $3$-class $2$ that do not arise as \Spsags. These include the abelian groups $\mathbb{Z}/3 \times \mathbb{Z}/9$ and $\mathbb{Z}/9 \times \mathbb{Z}/9$ which are examples of the pseudo-Schur+1 phenomenon discussed above. If one were enumerating tuples of relations in $X_2^3$ and constructing the corresponding quotients of $F_2$, then after encountering the groups $G_1$, $G_2$ and $G_3$ and computing the corresponding values of $\Meas_2$, one would be able to terminate the enumeration and deduce that the list of \Spsags of $3$-class $2$ is complete since
\[ 208/243 + 104/729 + 1/729 = 1 = \Meas_1(\mathbb{Z}/3 \times \mathbb{Z}/3). \]
Note that $\mathbb{Z}/3 \times \mathbb{Z}/3$ is the unique $2$-generated $3$-group of $3$-class $1$. All $2$-generated $3$-groups of larger $3$-class descend from this group.

These computations can be continued. For instance, $G_1 = \mathrm{SmallGroup}(27,3)$ has $11$ children, all of which have a GI-automorphism and all of which are \Spsags. Of these, $5$ are terminal \Spsgs and are the three examples of order $81$ together with the two \Ssgs of order $243$ arising at the start of this subsection. 
Computing $\Meas_3(Q)$ for each of the $11$ children and then summing we obtain $\Meas_2(G_1) = 208/243$, as expected. 
Continuing this process yields a probability distribution on part of O'Brien's rooted tree. The figure below shows more of the tree below SmallGroup$(27,3)$. Each node represents a descendant group $G$ labeled with $\Meas_c(G)$ where $c$ is the $p$-class $G$. The values of $\Meas(G)$ are not listed explicitly, but can be obtained by applying Theorem~\ref{thm-meas-relations}(i). Simply take the value on any node and subtract off the sum (if any) of the measures of its children. \\

\vfill

\pagebreak
\includegraphics[scale=0.62,angle=90]{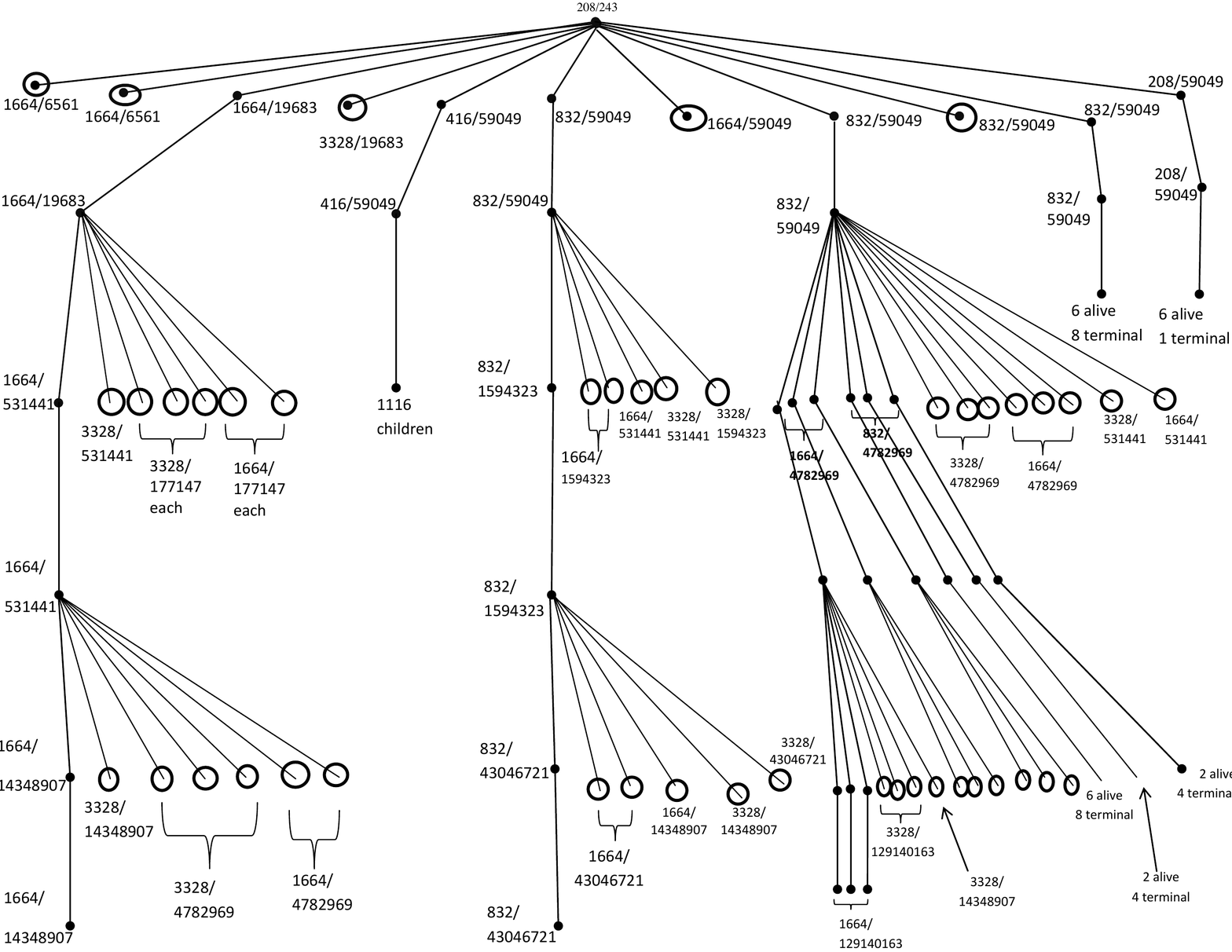}
\pagebreak


\section{Non-abelian Conjectures}\label{section-conjectures}

In this section, we introduce our main heuristic assumption, that the frequency of occurrence of certain groups as $G_K$ (or $G_K/P_c(G_K)$) as $K$ varies among real quadratic fields, is given by the group theoretical measures introduced in Section~2. 

To make this more precise, we introduce some notation. For $x>0$, let $\FF_x$ denote the set of real quadratic fields with discriminant not exceeding $x$.  If $K \in \FF_x$ then we let $A_K$ denote the $p$-class group of $K$ and $G_K$ denote the Galois group of the maximal unramified $p$-extension of $K$. For each natural number $g$, let $\FF_{x,g}$ be the subset of $\FF_x$ consisting of those fields $K$ having $d(G_K) = d(A_K) = g$.  
For pro-$p$ groups $G$ and $H$, define  $\ch_G(H)$ to be $1$ if $H \cong G$ and $0$ otherwise.
\begin{definition}
Let $G$ be a  finitely generated pro-$p$ group with generator rank $g$. We define 
\[ \Freq(G) = \lim_{x \to \infty} \frac{\sum_{K \in \FF_{x,g}}
\ch_G(\G_K)}{\sum_{K\in \FF_{x,g}} 1},
\]
assuming the limit exists. If $G$ is also finite then, for $c \geq 1$, we define
\[ \Freq_c(G) = \lim_{x \to \infty} \frac{\sum_{K \in \FF_{x,g}}
\ch_G(\G_K/P_c(\G_K))}{\sum_{K\in \FF_{x,g}} 1},
\]
assuming the limit exists. 
\end{definition}

We conjecture that the frequencies defined above exist and, more specifically, that
\begin{conjecture}\label{MainConjecture}
For every finite $p$-group $G$, we have 
\begin{eqnarray*}
 \Freq(G) &=&\Meas(G) \\
 \Freq_c(G) &=& \Meas_c(G).
\end{eqnarray*}
In particular, $\Freq(G) \neq 0$ if and only if $G$ is a \Spsg and $\Freq_c(G) \neq 0$ if and only if $G$ is a \Spsag  with $p$-class $c$ or $G$ is a \Spsg with $p$-class at most $c$. 
\end{conjecture}

\begin{remark}
We point out that as a consequence of Conjecture \ref{MainConjecture}, we expect every {\em finite} \Spsg to occur as $\G_K$  for a positive proportion of real quadratic fields $K$.  We also note that our conjectures in the non-abelian setting are compatible with those of Cohen and Lenstra. In particular, if $A$ is an abelian $p$-group, then one can define the frequency $\Freq^{\ab}(A)$ in an analogous way, as the asymptotic proportion of fields for which the $p$-class group is isomorphic to $A$. If $A$ has $p$-class $c$ and we fix $c' > c$, then using the definitions of the measures, Theorem~\ref{thm-meas-vs-meas-ab} and our conjecture above, we have
\[ \Freq^{\ab}(A) = \sum_G \Freq_{c'}(G) = \sum_G \Meas_{c'}(G) = \Meas_{c'}^{\ab}(A) = \Meas^{\ab}(A). \]
where the middle summations are over all \Spsags $G$ with $p$-class at most $c'$ and $G^{\ab} \cong A$. We are implicitly using the fact here that a field $K$ will have $A_K \cong A$ if and only if $G_K/P_{c'}(G_K)$ has abelianization isomorphic to $A$ once $c' > c$. By Theorem~\ref{main-group-theory-ab}, we then have
\[ \Freq^{\ab}(A) = \Meas^{\ab}(A) = \frac{1}{|\Aut(G)||G|}
\,
p^{g(g+1)}
\prod_{k=1}^g (1 - p^{-k}) \prod_{k=2}^{g+1} (1 - p^{-k}).
\]
This is consistent with the predictions made by Cohen and Lenstra in~\cite{CL2} although some manipulations are needed to extract this formula from their work. See the discussion in~\cite[Section~1]{BBH} for more details.
\end{remark}


\section{Index p Abelianization Data}\label{section-ipad-predictions}

To test our prediction of the last section, we would like to be able to compute the Galois group $G_K$ of the $p$-class tower of a given real quadratic field $K$ for many different choices of~$K$. This is hard to do, so we instead focus on collecting more limited information about  $G_K$, namely its abelianization and the abelianizations of its maximal subgroups. By class field theory, this can be done by computing the $p$-class groups of $K$ and its unramified extensions of degree $p$. As in~\cite{BBH}, we call this information the ``Index $p$ Abelianization Data" (or IPAD for short). 

For example, the groups SmallGroup$(81,i)$ for $i = 8,10$ each have IPAD $[[3,3]; [3,3]^3 [3,9]]$, which means that their abelianization is $[3,3]$ and that three of their maximal subgroups have abelianization $[3,3]$ and the other $[3,9]$. It turns out that these are the only \Spsgs with this IPAD. The measures of these groups add up to $8320/19683 = 0.4227$. We therefore expect that just over $42\%$ of real quadratic fields $K$ with $3$-class group of rank $2$ will have one of these two groups as their $G_K$. We call $8320/19683$ the measure of the IPAD.

In general, there may be infinitely many \Spsgs with a given IPAD $\mathcal{I}$, however, if we sum $\Meas_c(G)$ over the \Spsags $G$ of $p$-class at most $c$ with IPAD equal to  $\mathcal{I}$, then this quantity stabilizes for sufficiently large $c$ as explained in~\cite[Section 4]{BBH} and so we take this to be the definition of $\Meas(\mathcal{I})$. In practice, one can often avoid having to compute $\Meas_c(G)$ for large values of $c$ by recalling from~\cite{BBH} that there is a partial ordering on IPADs such that if $H$ is a child of $G$, then the IPAD of $G$ is less than or equal to that of $H$ and such that if their IPADs agree, then all further descendants have the same IPAD (we call such a branch stable). Using Theorem~\ref{thm-meas-relations}, we see that computing $\Meas_c(G)$ for the top node (of $p$-class $c$) in such a stable branch gives the part of the measure for this IPAD which arises from all of the \Spsgs within this branch of the tree.

We now illustrate the ideas above by determining the ten IPADs with largest measure.
\begin{theorem} 
\label{group_theory_predictions}

(1) IPAD $[[3,3]; [3,3]^3 [3,9]]$ has measure $8320/19683 = 0.4227$.

(2) IPAD $[[3,3]; [3,3]^3 [3,3,3]]$ has measure $1664/6561  = 0.2536$.

(3) IPAD $[[3,3]; [3,3]^3 [9,9]]$ has measure $3328/59049 = 0.0564$.

(4) IPAD $[[3,9]; [3,3,3] [3,9]^2 [3,27]]$ has measure $3328/59049 = 0.0564$.

(5) IPAD $[[3,3]; [3,3,3] [3,9]^3]]$ has measure $1664/59049 = 0.0282$.

(6) IPAD $[[3,3]; [3,3]^3 [9,27]]$ has measure $13312/531441 = 0.0250$.

(7) IPAD $[[3,9]; [3,3,9] [3,9]^3]]$ has measure $11648/531441 = 0.0219$.

(8) IPAD $[[3,9]; [3,3,3] [3,3,9] [3,9]^2]]$ has measure $3328/177147 = 0.0188$.

(9) IPAD $[[3,3]; [3,3,3]^2 [3,9]^2]]$ has measure $832/59049 = 0.0141$.

(10) IPAD $[[3,3]; [3,3,3]^3 [3,9]]$ has measure $832/59049 = 0.0141$.

\end{theorem}

\begin{proof}

The groups whose IPAD begins $[3,3]$ are the descendants of SmallGroup$(27,3)$, and so the reader is referred to the earlier figure displaying these. It has $11$ children, all of which are \Spsags,
and as noted above, $5$ are terminal. The 2nd and 4th of the $11$ have IPAD(1), as does the nonterminal 3rd child. Its only child, which is a \Spsag, has IPAD $[[3,3]; [3,3]^3 [9,9]]$. No other groups have small enough IPAD to produce IPAD(1) and so this establishes (1) above.

The child with IPAD $[[3,3]; [3,3]^3 [9,9]]$ has $7$ children of its own, $6$ terminal, of which $3$ contribute to IPAD(3) and $3$ to IPAD(6). The one nonterminal
child has a single \Spsag as a child. Its IPAD is $[3,3]; [3,3]^3 [27,27]]$. In this way, we exhaust all possibilities and so establish (3) and (6)
above. This branch appears to be following a simple pattern so that we conjecture $[[3,3]; [3,3]^3 [3^k,3^k]]$ will have measure $3328/3^{3k+4}$ and $[[3,3]; [3,3]^3 [3^k,3^{k+1}]]$ measure $13312/3^{3k+6}$ (for $k \geq 2$).

As for IPAD(2), SmallGroup$(81,7)$ is terminal and is the only one of the $11$ children whose IPAD only involves $3$s and so (2) is established.
As for IPAD(5), SmallGroup$(243,5)$ accounts for this. The $8$th child might also have contributed, since it has the same IPAD,
but it is not a \Spsg itself and only one of its children is a \Spsag and has larger IPAD, so that takes care of (5). This in turn has $14$ children, of which $8$ are terminal. Of
these $2$ contribute measure $1664/177147$ to IPAD $[[3,3]; [3,3,3] [3,9]^2 [9,9]]$ and the remaining $6$, plus $4$ of the nonterminal
children whose subsequent branches are stable, measure $6656/1594323$ to IPAD $[[3,3]; [3,3,3] [3,9]^2 [9,27]]$. In both cases,
this is too small to make the top ten list above, as is the remaining measure once this is accounted for.

SmallGroup$(243,7)$ accounts for IPAD(9). The $5$th child might also have contributed since it has the same IPAD, but it is not a \Spsg itself and only one of its children is a \Spsag and has larger IPAD.This in turn has $1116$ children which are \Spsags.  Their IPADs do not make the top ten list above. As for IPAD(10),
the $6$th child and its only child which is a \Spsag have this, showing that this is a stable branch.

The $10$th and $11$th children of SmallGroup$(27,3)$ both have IPAD $[3,3]; [3,9]^4]$. The $11$th (but not the $10$th) leads to a
stable branch, but this only yields measure $208/59049$ for that IPAD which is too small to make our list. The $10$th produces some terminal groups with IPADs
$[[3,3]; [3,9]^3 [9,9]],$ $ [[3,3]; [3,9]^3 [9,27]]$, and so on, but their measures do not make our list.

As for the IPADs starting $[3,9]$, these must come from descendants of SmallGroup$(81,3)$, which has $31$ children, all \Spsags.
The $8$th and $9$th of these are terminal (and are SmallGroup$(243,i)$ for $i=19, 20$) and account for IPAD(4). No other child
has small enough IPAD to contribute. The $5$th child (SmallGroup$(243,16)$) is terminal and contributes $3328/177147$ to IPAD(7). The $3$rd
and $4$th children also have the same IPAD. The branch of the $3$rd child is stable (and terminates soon after) so it contributes its measure, $1664/531441$,
whereas the $4$th group is not a \Spsg and its only child which is a \Spsag has larger IPAD and so contributes nothing.

The $7$th child  (SmallGroup$(243,18)$) is terminal and accounts for IPAD(8). The $6$th child might also have contributed, but it is not a \Spsg itself and only one 
of its children is a \Spsag and has larger IPAD. There are two other terminal children of SmallGroup$(81,3)$, namely SmallGroup$(729,i)$ for $i=14,15$. From these,
$[[3,9]; [3,3,9]^2 [3,27]^2]]$ acquires measure $3328/531441$, not enough to make the list. Another child  (SmallGroup$(729,13)$) which is not a \Spsg has the same IPAD, but
its children have larger IPADs.

None of the other children of  SmallGroup$(81,3)$ or of the third group of $3$-class $2$ produce IPADs with measure
large enough to make the list.
\end{proof}


\section{Numerical Data}\label{section-data}
As evidence for our conjectures we have collected numerical data in the case of the smallest odd prime $p = 3$ and generator rank $g = 2$. In particular, we have computed the four unramified cyclic extensions of degree $3$ over $K$ and their $3$-class groups (assuming GRH) for all real quadratic fields $K$ with $3$-class group of rank $2$ and discriminant $d_K$ satisfying $d_K < 10^9$.  By class field theory, this yields the IPAD for the Galois group $G_K$ for each of these fields. The calculations were carried out indirectly by using existing methods to enumerate non-cyclic cubic extension of $\mathbb{Q}$. See~\cite[Section~5]{BBH} for more details.

The computations were implemented  using both the symbolic algebra package
\pari~\cite{pari}, version 2.5.4 and \magma~\cite{magma}, version~2.19-5 running on $2 \times 2.66$ GHz 6-Core Intel Xeon processors under OS X 10.8.5. 
The computations were run in parallel across multiple cores by dividing up the discriminants into subintervals and then searching through a space of potential defining polynomials. Roughly $3000$ core hours were used in total.

We now give tables summarizing the data collected. In each table, we have broken down the interval of discriminants $d_K$ with $1 < d_K <  10^9$ into five nested subintervals $I_j$  where $I_j = \{ d_K \mid  1 \leq d_K \leq j \cdot 10^7  \}$ and we have selected values of $j$ so that the length of each successive subinterval is scaled by a factor of $\sqrt{10} \approx 3.2$.

The first table is a census of the most common IPADs. The second lists their relative proportions obtained by dividing through by the total number of fields examined in each column. In addition, the last column of the second table lists the values predicted by our heuristics as computed in Theorem~\ref{group_theory_predictions}.  As in \cite{BBH}, there are two IPADs which each determine the isomorphism type of a single group. These appear in lines 5 and 9 of Table~2 and correspond to the groups SmallGroup(243,5) and SmallGroup(243,7) respectively.  Thus, on these two lines, the predicted and computed frequencies for an individual group can be compared, providing a direct test of our non-abelian heuristics.


{\footnotesize
\begin{table}[htp]
\caption{Census of the most common IPADs.}
\begin{center}

\begin{tabular}{|l |c|c|c|c|c|}
\hline
 & $ I_1$ & $ I_{3.2}$ & $ I_{10}$ & $ I_{32}$ & $ I_{100}$\\
\hline
$[3,3]$; $[3,3]^3$ $[3,9]$   & 1382 & 5035 & 17618 & 61826 & 208236 \\
\hline
$[3,3]$; $[3,3]^3$ $[3,3,3]$   & 698 & 2813 & 10244 & 36285 & 122955 \\
\hline
$[3,3]$; $[3,3]^3$ $[9,9]$   & 150 & 623 & 2180 & 7869 & 26678 \\
\hline
$[3,9]$; $[3,3,3]$ $[3,9]^2$ $[3,27]$   & 135 & 541 & 2141 & 7831 & 26748 \\
\hline
$[3,3]$; $[3,3,3]$ $[3,9]^3$   & 93 & 323 & 1122 & 3993 & 13712 \\
\hline
$[3,3]$; $[3,3]^3$ $[9,27]$   & 72 & 242 & 955 & 3444 & 11780 \\
\hline
$[3,9]$; $[3,3,9]$ $[3,9]^3$   & 45 & 211 & 805 & 2970 & 10373 \\
\hline
$[3,9]$; $[3,3,3]$ $[3,3,9]$ $[3,9]^2$   & 32 & 164 & 718 & 2535 & 8733 \\
\hline
$[3,3]$; $[3,3,3]^2$ $[3,9]^2$   & 47 & 156 & 546 & 1987 & 6691 \\
\hline
$[3,3]$; $[3,3,3]^3$ $[3,9]$   & 27 & 123 & 493 & 1901 & 6583 \\
\hline
Other IPADs (175 types)  & 189 & 778 & 2969 & 11142 & 39267 \\
\hline
\hline
Total   & 2870 & 11009 & 39791 & 141783 & 481756 \\
\hline
\end{tabular}
\end{center}
\end{table}
}


{\footnotesize
\begin{table}[htp]
\caption{Relative proportions of the most common IPADs.}
\label{last-table}
\begin{center}

\begin{tabular}{|l |c|c|c|c|c|c|}
\hline
 & $ I_1$ & $ I_{3.2}$ & $ I_{10}$ & $ I_{32}$ & $ I_{100}$& Predicted \\
\hline
$[3,3]$; $[3,3]^3$ $[3,9]$   & 0.4815 & 0.4574 & 0.4428 & 0.4361 & 0.4322 & 
0.4227 \\
\hline
$[3,3]$; $[3,3]^3$ $[3,3,3]$   & 0.2432 & 0.2555 & 0.2574 & 0.2559 & 0.2552 & 
0.2536 \\
\hline
$[3,3]$; $[3,3]^3$ $[9,9]$   & 0.0523 & 0.0566 & 0.0548 & 0.0555 & 0.0554 & 
0.0564 \\
\hline
$[3,9]$; $[3,3,3]$ $[3,9]^2$ $[3,27]$   & 0.0470 & 0.0491 & 0.0538 & 0.0552 & 
0.0555 & 0.0564 \\
\hline
$[3,3]$; $[3,3,3]$ $[3,9]^3$   & 0.0324 & 0.0293 & 0.0282 & 0.0282 & 0.0285 & 
0.0282 \\
\hline
$[3,3]$; $[3,3]^3$ $[9,27]$   & 0.0251 & 0.0220 & 0.0240 & 0.0243 & 0.0245 & 
0.0250 \\
\hline
$[3,9]$; $[3,3,9]$ $[3,9]^3$   & 0.0157 & 0.0192 & 0.0202 & 0.0209 & 0.0215 & 
0.0219 \\
\hline
$[3,9]$; $[3,3,3]$ $[3,3,9]$ $[3,9]^2$   & 0.0111 & 0.0149 & 0.0180 & 0.0179 & 
0.0181 & 0.0188 \\
\hline
$[3,3]$; $[3,3,3]^2$ $[3,9]^2$   & 0.0164 & 0.0142 & 0.0137 & 0.0140 & 0.0139 & 
0.0141 \\
\hline
$[3,3]$; $[3,3,3]^3$ $[3,9]$   & 0.0094 & 0.0112 & 0.0124 & 0.0134 & 0.0137 & 
0.0141 \\
\hline
Other IPADs (175 types)  & 0.0659 & 0.0707 & 0.0746 & 0.0786 & 0.0815 & 0.0888\\
\hline
\end{tabular}
\end{center}
\end{table}
}

\noindent{\sc Acknowledgements.}  We acknowledge useful correspondence 
and conversations with Melanie Matchett Wood. 

\pagebreak


\def\Sset{\Omega}   
\def\Ssetab{\Omega^{\ab}}
\def\hatS{\hat{\Sset}}
\def\hatSab{\hat{\Sset}^{\ab}}


\begin{center}
\begin{Large}
{Appendix: Measures for infinite groups}
\end{Large}
\end{center}

In the current paper and also in~\cite{BBH}, we initially avoided the issue of assigning measures to groups which are infinite. Given the limitations one immediately encounters when trying to test the conjectures as stated, this seems like a small omission. On the other hand, it would be nice to have a consistent theoretical framework for assigning probabilities to all groups that may arise, even if testing the conjectures for individual infinite groups seems out of reach currently. We will now show how this can be carried out for \Spsgs\!\!. The same ideas can be easily modified to extend the measure introduced for \Ssgs in~\cite{BBH}.

Fix an odd prime $p$ and positive integer $g$.
Let $\Sset$ denote the set of all \Spsgs with generator rank $g$ (up to isomorphism). For $c \geq 1$, let $\Sset_c = \{ G_c \mid G \in \Sset \}$. Note that although the set $\Sset$ may be infinite and contain infinite pro-$p$ groups as elements, the set $\Sset_c$ is always finite and contains only finite $p$-groups with $p$-class at most $c$.  In Section~2,  we have introduced a function $\Meas(G)$ which is defined only for the {\em finite} groups $G$ in $\Sset$. We will now show that this function can be extended to cover all elements of $\Sset$ by using some standard results in measure theory to define a measure on an appropriate $\sigma$-algebra. (Here we encounter an unfortunate notational conflict since this use of $\sigma$ has nothing to do with the $\sigma$-automorphisms of the groups involved.)

We start by considering the functions $\Meas_c: \Sset_c \rightarrow [0,1]$ for $c \geq 1$ that were  introduced in Section~2. Each of these functions can be extended from individual elements to subsets by summation and thus each gives rise to a measure defined on the (finite) power set algebra $\mathcal{P}(\Sset_c)$. We have a natural map $X_c^{g+1} \rightarrow \Sset_c$ defined by $v \mapsto F_c / \langle v \rangle$. This map is surjective (by Lemma~2.5) and $\Meas_c$ is simply the probability measure which results from pushing forward the uniform counting measure for $X_c^{g+1}$ along this map. 

We now focus our attention on these uniform counting measures. For each $c$, we have a natural map $\psi_c: X^{g+1} \rightarrow X_c^{g+1}$. Let $\mathcal{A}_c \subseteq \mathcal{P}(X^{g+1})$ denote the algebra which results by taking the inverse image of the algebra $\mathcal{P}(X_c^{g+1})$ under $\psi_c$. Let $\mu_c: \mathcal{A}_c \rightarrow [0,1]$ then denote the probability measure that results from pulling back the uniform counting measure along $\psi_c$. i.e. for $A \in \mathcal{A}_c$, we have
\[  \mu_c(A) =   \frac{|\psi_c(A)|}{|X_c|^{g+1}}.
\]
Observe that the family of algebras $\{ (\mathcal{A}_c, \mu_c) \}_{c =1}^\infty$ form a nested sequence inside $\mathcal{P}(X^{g+1})$. The associated probability measures  $\mu_c$ are compatible in the sense that if $A \in \mathcal{A}_c \subseteq \mathcal{A}_{c+1}$ then $\mu_{c+1}(A) = \mu_c(A)$. This follows from the fact that the fibers of the natural projection $X_{c+1}^{g+1} \rightarrow X_c^{g+1}$ are uniform in size. This is explained in the proof of \cite[Theorem~2.11]{BBH}. It is due to the fact that this projection fits into a commuting square in which the other maps also have constant fibers. This observation will recur in our discussion of Figures~\ref{comm-sq1} and~\ref{comm-cube} below.

Define $\mu:  \cup_{c = 1}^\infty \mathcal{A}_c \rightarrow [0,1]$ by $\mu(A) = \mu_c(A)$ if $A \in \mathcal{A}_c$. The map $\mu$ is well defined by our previous observation. Now let $\mathcal{A}$ be the smallest $\sigma$-algebra containing $\cup_{c = 1}^\infty \mathcal{A}_c$. By the Carath\'eodory Extension Theorem, the map $\mu$ can be extended to $\mathcal{A}$ provided that the following conditions hold:
\begin{itemize}
\item[(i)] $\mu(\emptyset) = 0$.
\item[(ii)] If $A \in \cup_{c = 1}^\infty \mathcal{A}_c$ and we have $A = \cup_{i=1}^\infty A_i$ where $\{ A_i \}$ is a collection of pairwise disjoint elements also in $\cup_{c = 1}^\infty \mathcal{A}_c$, then
\[ \mu(A) = \sum_{i=1}^\infty \mu(A_i). \]
\end{itemize}
The fact that Condition~(i) holds is clear. Condition~(ii) can be reformulated as:
\begin{itemize}
\item[$\text{(ii)}^\prime$] If $\{ B_i \}$ is a sequence of elements in $\cup_{c = 1}^\infty \mathcal{A}_c$ satisfying $B_i \supseteq B_{i+1}$ for $i \geq 1$ and $\cap_{i = 1}^\infty B_i = \emptyset$ then
\[   \lim_{i \rightarrow \infty} \mu(B_i) = 0. \]
\end{itemize}
This  reformulation can now be verified using topological considerations. The set $X$ is closed inside the free pro-$p$ group $F$ and hence both it and the product space $X^{g+1}$ are compact.  The maps from $X^{g+1}$ to the finite discrete spaces $X_c^{g+1}$ are continuous so it follows that all the elements of the algebras $\mathcal{A}_c$ and hence of $\cup_{c = 1}^\infty \mathcal{A}_c$ are compact inside $X^{g+1}$. If $\{ B_i \}$ is a sequence of elements in $\cup_{c = 1}^\infty \mathcal{A}_c$ satisfying $B_i \supseteq B_{i+1}$ for $i \geq 1$ and $\cap_{i = 1}^\infty B_i = \emptyset$, then using compactness we see that some finite intersection must be empty. i.e. there exists $n \in \mathbb{N}$ such that
\[ \emptyset = \bigcap_{i = 1}^n B_i = B_n. \]
But then,  for all $i \geq n$, we have $B_i = B_n = \emptyset$ and $\mu(B_i) = 0$ which implies  $\lim_{i \rightarrow \infty} \mu(B_i) = 0$ as desired.

Having established that $\mu$ can be extended to $\mathcal{A}$, we now wish to use $\mu$ to define a measure on an associated space of groups. Let $\hatS = \{ F/\langle v \rangle \mid v \in X^{g+1} \}$. Note that $\hatS$ is strictly larger than $\Sset$ since a quotient $F/\langle v \rangle$ may have relation rank strictly less than both $g$ and $g+1$. As an extreme example, if every component in $v$ is the identity element then $F/\langle v \rangle \cong F$ showing that $\hatS$ contains the free group $F$. We will return to this issue shortly.

Let $\eta: X^{g+1} \rightarrow \hatS$ be the map $v \mapsto F/\langle v \rangle$. Defining $E \subseteq \hatS$ to be measurable if $\eta^{-1}(E) \in \mathcal{A}$ we obtain a $\sigma$-algebra $\hat{\mathcal{B}}$ on $\hatS$. We can then push $\mu$ forward along $\eta$ to obtain a measure $\Meas: \hat{\mathcal{B}} \rightarrow [0,1]$ by defining $\Meas(E) = \mu(\eta^{-1}(E))$. 
Observe that if $G \in \hatS$ then the singleton set $\{ G \}$ belongs to $\hat{\mathcal{B}}$. This follows from the fact that $G \cong \varprojlim G_c$. In more detail, if we let $f_c: \hatS \rightarrow \Sset_c$ denote the natural map which sends $G \mapsto G_c$ then one can check that this function is measurable with respect to the respective $\sigma$-algebras $\hat{\mathcal{B}}$ and $\mathcal{P}(\Sset_c)$. It then follows that $\{ G \} = \cap_{c=1}^\infty f_c^{-1}(G_c) \in \hat{\mathcal{B}}$ and we see that
\[ \Meas(\{ G \}) = \lim_{c \rightarrow \infty} \mu(f_c^{-1}(G_c)) =  \lim_{c \rightarrow \infty} \Meas_c(G_c). \]
In particular, if $G$ is finite then we have $G \cong G_c$ once $c$ is sufficiently large and we obtain the same value for $\Meas(G) := \Meas(\{ G \})$ as specified in Definition~\ref{def-meas}. Thus this new definition of $\Meas$ extends the old one.

We finish by showing that $\Sset \in \hat{\mathcal{B}}$ and that $\Meas(\Sset) = 1$, equivalently, the complement $\hatS - \Sset \in \hat{\mathcal{B}}$ and $\Meas(\hatS - \Sset) = 0$. In particular, if $G \in \hatS - \Sset$ then $\Meas(G) = 0$. This is  desirable for our applications since the groups in $\hatS - \Sset$ should never arise as Galois groups of the extensions we are considering. If we then define $\mathcal{B} = \{ E \cap \Sset \mid E \in \hat{\mathcal{B}} \}$ and restrict $\Meas$ to $\mathcal{B}$, we obtain a probability measure on $\Sset$. As noted above, this definition of $\Meas$ extends our earlier definition which was restricted to individual finite groups in $\Sset$.

First, we note that all of the constructions above can be carried out in the abelian setting. Recall from Section~\ref{subsection-ab-measure} that $X^{\ab} = \Phi(F^{\ab})$. As before, we construct an algebra $\cup_{c = 1}^\infty \mathcal{A}_c^{\ab}$ and then $\sigma$-algebra $\mathcal{A}^{\ab}$ on $(X^{\ab})^{g+1}$ with accompanying measure $\mu^{\ab}$. One can then see that the natural reduction map $X^{g+1} \rightarrow (X^{\ab})^{g+1}$ is measurable and compatible with the measures $\mu$ and $\mu^{\ab}$. This follows ultimately from the observation that the square in Figure~\ref{comm-sq1} commutes and that all of the maps between the finite sets appearing there have fibers which are constant in size (for all $c \geq 1$). This in turn follows since this square is the front face of the cube in Figure~\ref{comm-cube}. Note that the maps on the back face of the cube are all induced by natural epimorphisms either from $F_c$ to $F_c^{\ab}$ or $F_{c+1}$ to $F_c$. The maps $\phi_\ast$ connecting the front and back faces are not homomorphisms, but do have fibers of constant size as discussed in Lemma~\ref{phi-map1}. The $g+1$ components of each map $\phi_\ast$ have the form $t \mapsto t^{-1} \sigma(t)$. For the abelian objects, this simplifies to $t \mapsto t^{-2}$.

\begin{figure}
\begin{tikzcd}
X_{c+1}^{g+1} \ar{rr}{} \ar{dd}{} 
  & 
  &(X_{c+1}^{\ab})^{g+1} \ar{dd}{}\\
& & \\
X_{c}^{g+1} \arrow{rr}{} 
  & 
  &(X_{c}^{\ab})^{g+1}
\end{tikzcd}
\caption{}
\label{comm-sq1}
\end{figure}

\begin{figure}
\begin{tikzcd}
\Phi(F_{c+1})^{g+1} \ar{dd} \ar{rd}[sloped, near start]{\phi_{c+1}} \ar{rr} 
  & 
   &\Phi(F_{c+1}^{\ab})^{g+1} \ar{dd} \ar{rd}[sloped, near start]{\phi_{c+1}^{\ab}} \\
&X_{c+1}^{g+1} \ar[crossing over]{rr}  
  & 
  & (X_{c+1}^{\ab})^{g+1} \ar{dd}\\
\Phi(F_{c})^{g+1} \ar{rd}[sloped, near start]{\phi_{c}} \ar{rr} 
  & 
  &\Phi(F_{c}^{\ab})^{g+1} \ar{rd}[sloped, near start]{\phi_{c}^{\ab}} \\
&X_{c}^{g+1} \ar[crossing over, leftarrow]{uu}\ar{rr} 
  & 
  &(X_{c}^{\ab})^{g+1}
\end{tikzcd}
\caption{}
\label{comm-cube}
\end{figure}

Now define the corresponding space of groups in the abelian setting by
\[ \hatSab = \{ F^{\ab}/\langle v \rangle \mid v \in (X^{\ab})^{g+1} \} = \{ G^{\ab} \mid G \in \hatS\} \] 
and let $\eta: (X^{\ab})^{g+1} \rightarrow \hatSab$ be the map $v \mapsto F^{\ab}/\langle v \rangle$. Defining $E \subseteq \hatSab$ to be measurable if $\eta^{-1}(E) \in \mathcal{A}^{\ab}$ we obtain a $\sigma$-algebra $\hat{\mathcal{B}}^{\ab}$ on $\hatSab$. We can then push $\mu^{\ab}$ forward along $\eta$ to obtain a measure $\Meas^{\ab}$ on $\hat{\mathcal{B}}$ by defining $\Meas^{\ab}(E) = \mu^{\ab}(\eta^{-1}(E))$. As with the measure $\Meas$ on $\hatS$, this definition of $\Meas^{\ab}$ extends the one given in Definition~\ref{def-meas-ab}.

Let $\alpha: \hatS \rightarrow \hatSab$ be the map  $G \mapsto G^{\ab}$.  Define $\Ssetab = \alpha(\Sset)$ and observe that $\Ssetab$ consists of all {\em finite} abelian $p$-groups. Further, $\alpha^{-1}(\Ssetab) = \Sset$. This follows since if $G \in \hatS - \Sset$, then $r(G) < g$ and $G^{\ab}$ must have at least one infinite cyclic component $\mathbb{Z}_p$. The map $\alpha$ is surjective and measurable and we have $\Meas(\alpha^{-1}(E)) = \Meas^{\ab}(E)$ for all $E \in \hat{\mathcal{B}}^{\ab}$ since $\alpha$ forms part of the commuting square in Figure~\ref{comm-sq2}. Thus to show $\Sset$ is measurable with $\Meas(\Sset) = 1$, it suffices to show that $\Ssetab$ is measurable with $\Meas^{\ab}(\Ssetab) = 1$.

\begin{figure}
\begin{tikzcd}
X^{g+1} \ar{rr}{\eta} \ar{dd}{} 
  & 
  &\hatS \ar{dd}{\alpha}\\
& & \\
(X^{\ab})^{g+1} \arrow{rr}{\eta} 
  & 
  &\hatSab
\end{tikzcd}
\caption{}
\label{comm-sq2}
\end{figure}

By definition, this reduces to verifying that $\eta^{-1}(\Ssetab) \subseteq (X^{\ab})^{g+1}$ is measurable with measure $1$ under $\mu^{\ab}$. This can be seen by first noting that
$(X^{\ab})^{g+1} = \Phi(F^{\ab})^{g+1}$ is a compact abelian group. The measure $\mu^{\ab}$ is translation invariant and hence is a Haar measure, normalized so that $\mu^{\ab}( \Phi(F^{\ab})^{g+1}) = 1$. Since $F^{\ab} \cong \mathbb{Z}_p^g$, we have $\Phi(F^{\ab}) \cong (p \mathbb{Z}_p)^g  \cong \mathbb{Z}_p^g$, and so the elements of $\Phi(F^{\ab})^{g+1}$ can be viewed as $(g+1) \times g$ matrices with entries in $\mathbb{Z}_p$. In particular, the elements of $\eta^{-1}(\Ssetab)$ are the matrices of full rank and it is a standard fact that these have measure $1$ with respect to this Haar measure.

\end{document}